\documentclass[11pt]{article}
\usepackage{amsmath, amsthm, amssymb}
\usepackage{enumerate}
\usepackage[a4paper]{geometry}
\usepackage[colorlinks,urlcolor=blue,linkcolor=blue,citecolor=blue,final,plainpages=false,pdfpagelabels,pdfdisplaydoctitle]{hyperref}

\theoremstyle{plain}
\newtheorem{thm}{Theorem}
\newtheorem{prp}[thm]{Proposition}
\newtheorem{lem}[thm]{Lemma}
\newtheorem{rem}[thm]{Remark}
\newtheorem{defn}[thm]{Definition}

\newcommand\R{\mathbb{R}}
\newcommand\N{\mathbb{N}} 

\hypersetup{%
  pdftitle={Global regularity for a logarithmically supercritical hyperdissipative dyadic equation},
  pdfauthor={D. Barbato, F. Morandin, M. Romito}
}

\title{Global regularity for a logarithmically supercritical hyperdissipative dyadic equation}

\author{
  D.~Barbato\\
    \small Universit\`a di Padova\\[-.2em]
    \scriptsize\texttt{barbato@math.unipd.it}
  \and
  F.~Morandin\\
    \small Universit\`a di Parma\\[-.2em]
    \scriptsize\texttt{francesco.morandin@unipr.it}
  \and
  M.~Romito\\
    \small Universit\`a di Pisa\\[-.2em]
    \scriptsize\texttt{romito@dm.unipi.it}
}

\begin{document}

\hypersetup{pageanchor=false}
\maketitle
\hypersetup{pageanchor=true}

\begin{abstract}
  We prove global existence of smooth
  solutions for a slightly supercritical dyadic model. We consider a
  generalized version of the dyadic model introduced by
  Katz-Pavlovic~\cite{KatPav04} and add a viscosity term with critical
  exponent and a supercritical correction. 
  This model
  catches for the dyadic a conjecture that for Navier-Stokes equations
  was formulated by Tao~\cite{Tao2009}.
\end{abstract}
\section{Introduction}      

The \emph{a priori} estimate of relevant quantities is a crucial
part of the analysis of PDEs. For our purposes, the most interesting
example are the Navier--Stokes equations in dimension three. In that
case the kinetic energy and the energy dissipation are super--critical,
hence in a way negligible, quantities with respect to the scaling
invariance of the problem.
Indeed, proofs of regularity are
available only in the so--called hyper--dissipative case,
where the Laplace operator is replaced by $(-\Delta)^\alpha$
for $\alpha\geq5/4$ and this
additional dissipation makes the energy relevant again
(see for instance~\cite{MatSin1999,KatPav2002}).

In a recent paper Tao~\cite{Tao2009}
has shown that hyper--dissipativity can be slightly relaxed
by a logarithmic factor. The idea originates from the same author~\cite{Tao2007} and has been applied in other problems, mainly from
dispersive equations.
In~\cite{Tao2009} Tao adds a small correction to the hyper--dissipative
term, replacing $(-\Delta)^{5/4}$ with
\[
  \frac{(-\Delta)^{{5/4}}}{g((-\Delta)^{1/2})^2},
\]
and provides a simple and neat proof of global existence if
$\int 1/(sg(s)^4)=\infty$. He then suggests that
the same result should hold, based on some heuristics on the flow
of energy, under the weaker condition $\int 1/(sg(s)^2)=\infty$.

The aim of this paper is to prove Tao's conjecture for the dyadic
model, a simplified version of the Navier--Stokes equations, that
nevertheless has shown to be an effective tool in the understanding of
the full Navier--Stokes problem~\cite{Tao2014}. In particular, we
believe that the main result of our paper (Theorem~\ref{thm:main})
gives a complete answer to some questions raised in Remark 5.2
of~\cite{Tao2014}. As a bonus result, in
Section~\ref{s:generalization_to_nonconstant_phi} we prove that the
conjecture in~\cite{Tao2009} is true for the vector-valued dyadic
model introduced in~\cite{Tao2014}.  A proof of the conjecture for the
full Navier--Stokes equations is a work in progress.
\subsection{The model}

Given $\beta>0$ and two real sequences $\phi=(\phi_n)_{n\geq1}$
and $g = (g_n)_{n\geq1}$, with $\phi$ bounded and $g$ positive,
set $k_n=2^{\beta n}$ for $n\geq0$. Consider the
critical hyper--dissipative generalized dyadic model,
\begin{equation}\label{e:main}
  \begin{cases}
    X_n'
      = \phi_{n-1}k_{n-1}X_{n-1}^2-\phi_nk_nX_nX_{n+1}-\frac1{g_n}k_nX_n,\\
    X_n(0)
      = x_n,
  \end{cases}
  \qquad t>0, n\geq1,
\end{equation}
where $X=(X_n)_{n\geq0}$ is a family of real functions,
$X_0\equiv0$ and $x=(x_n)_{n\geq1}$ is the given initial
condition. 

The classical critical regime here corresponds to $g\equiv1$.
Tao's statement for Navier-Stokes equation, transposed on
our model, works whenever $\sum_ng_n^{-2}=\infty$
($g_n=\sqrt n$ for instance), while the conjecture, on our model,
states that global regularity should hold for
$\sum_ng_n^{-1}=\infty$ (e. g. $g_n=n$).

The role of the coefficients $\phi$ is to break the structure
 of the
non--linearity. Otherwise, as shown in~\cite{BarMorRom2011},
 if
$\phi\equiv1$, the energy flow is very steady, in the sense that the
transfer of energy from $X_n$ to $X_{n+1}$ starts before $X_{n-1}$ is
discharged enough and this allows to prove regularity in a full
supercritical regime. Further generalizations
are possible, see Section~\ref{s:remarks}.

The dyadic model has been introduced in~\cite{KatPav04} and
analyzed in several other works~\cite{CheFriPav2007,CheFriPav2010,BarFlaMor2010CRAS,BarFlaMor2011TAMS}.
The model with viscosity has been initially introduced
in~\cite{Cheskidov08} and further analyzed in~\cite{CheFri09,BarMorRom2011}.
\subsection{The dyadic version of~\texorpdfstring{\cite{Tao2009}}{[Tao2009]}}

It is easy to be convinced that Tao's condition $\int1/(sg(s)^4)=\infty$
reads in our case as $\sum 1/g_n^2=\infty$. To this end, we reproduce
in this section the idea of~\cite{Tao2009} adapted to the dyadic
framework. Assume also, as we do, that $(g_n)_{n\geq1}$ and
$(k_n/g_n)_{n\geq1}$ are non--decreasing. Assume moreover, for
simplicity, that $g_n = g(n)$, where $g$ is non--decreasing,
continuous, non--zero on $[0,\infty)$ and $\int g(x)^{-2} = \infty$.

Given a solution $X$, set for $s\geq1$,
\[
  a(t)
    = \sum_{n=1}^\infty \frac{k_n}{g_n}X_n^2,
      \qquad
  A(t)
    = \sum_{n=1}^\infty k_n^{2s}X_n^2,
      \qquad
  B(t)
    = \sum_{n=1}^\infty \frac{k_n^{2s+1}}{g_n}X_n^2.
\]
We know by the energy estimate that $a\in L^1([0,\infty))$.
By differentiating and using the Cauchy--Schwartz and Young inequalities,
\[
  \frac{d}{dt}A + 2B
    = 2(2^{2\beta s}-1)\sum_{n=1}^\infty \phi_n k_n^{2s+1}X_n^2X_{n+1}
    \leq B + c\sum_{n=1}^\infty g_nk_n^{2s+1}X_n^2X_{n+1}^2.
\]
Split the sum on the right--hand side in a sum $\mathtt{[L]}$ up to $N$
and in a sum $\mathtt{[H]}$ from $N$ on, where $N$ will be chosen at the
end. On the one hand,
\[
  \mathtt{[L]}
    = \sum_{n=1}^N g_n^2 \Bigl(\frac{k_n}{g_n}X_n^2\Bigr)(k_n^{2s}X_{n+1}^2)
    \leq cg_N^2 a A,
\]
on the other hand
\[
  \mathtt{[H]}
    = \sum_{n\geq N} \frac{g_n}{k_n}(k_n^{s+1}X_n^2)(k_n^{s+1}X_{n+1}^2)
    \leq \frac{g_N}{k_N}A^2.
\]
If we choose $N$ so that $k_N\approx A$, that is $N\approx\log A$,
we have
\[
  \dot A
    \leq c(1+a)g(\log A)^2 A,
\]
whose solutions stay bounded on bounded sets.
\subsection{The dyadic version of Tao's conjecture}

We present here a heuristic argument that shows, as in Remark 1.2
of~\cite{Tao2009}, that the weaker assumption $\sum_n g_n^{-1}=\infty$
is sufficient for global regularity.

Indeed, let $X$ be a weak solution on $[0,T)$ and consider a blow--up
scenario in $T$: at some time $t$ the energy of solution is concentrated
in $n,n+1,\ldots n+m$ and  $n\to\infty$ when $t\to T$. The balance of
energy on $n,\ldots,n+m$ yields:
\[
  \frac{d}{dt}\Bigl(\frac12\sum_{i=n}^{n+m} X_i^2\Bigr)
    = \phi_{n-1}k_{n-1}X_{n-1}^2X_n
      - \phi_mk_mX_{n+m}^2X_{n+m+1}
      - \sum_{i=n}^{n+m} \frac{k_i}{g_i}X_i^2,
\]
where we could imagine $\phi_{n-1}k_{n-1}X_{n-1}^2X_n$ as the
energy moving from $n-1$ to $n$, $\phi_mk_mX_{n+m}^2X_{n+m+1}$
the energy moving from $n+m$ to $n+m+1$, and $\frac{k_i}{g_i}X_i^2$
the energy dissipated in $i$. So, roughly speaking, $k_nX_n^3$
is the speed at which the energy moves from $n$ to $n+1$, whereas
$\frac{k_n}{g_n}X_n^2$ is the speed at which the energy is dissipated
in $n$.

Now in the blow--up scenario, to go to high ``$n$''s, the energy has to go
through all the states. The ratio between the energy dissipated and
the energy that goes through $n$ is $\frac{1}{g_nX_n}\geq\frac{C}{g_n}$.
So, to have a non--trivial amount of energy reaching the infinite state,
we have to require $\sum g_n^{-1}<\infty$.

Our proof is a rigorous version of the above argument. We find a recursive
formula~\eqref{e:dn_bound_claim} for the tail energy and dissipation.
Then we prove that any sequence satisfying the recursion decays
super--exponentially fast.
\section{Preliminaries}

\subsection{Basic definitions}

\begin{defn}
A weak solution is a sequence of $X=(X_n)_{n\geq1}$ of differentiable
functions on all $[0,\infty)$, satisfying~\eqref{e:main}.
\end{defn}

Whenever $X$ denotes a weak solution, $E_n(t)$ and $F_n(t)$ will
denote the energy of the tails: for all $n\geq1$ and $t\geq0$,
\[
E_n(t)
:=\sum_{i\leq n}X_i^2(t)
<\infty
,\qquad\text{and}\qquad
F_n(t)
:=\sum_{i\geq n}X_i^2(t)
\leq\infty.
\]
We will also denote by $E$ the total energy of the solution $X$: for
all $t\geq0$,
\[
E(t)
:=\sum_{n\geq 1}X_n^2(t)
=\lim_{n\to\infty}E_n(t)
=\|X(t)\|_H^2.
\]
Clearly
$E(t)=E_n(t)+F_{n+1}(t)$ for all $n\geq1$.
From~\eqref{e:main} we get
\[
\frac d{dt}(X_n^2)=2\phi_{n-1}k_{n-1}X_{n-1}^2X_n-2\phi_nk_nX_n^2X_{n+1}-\frac2{g_n}k_nX_n^2,
\]
so that if $X$ is a weak solution, for all $n\geq1$,
\begin{equation}\label{e:E_n'}
E_n'
=-2\phi_nk_nX_n^2X_{n+1}-\sum_{i\leq n}\frac2{g_i}k_iX_i^2.
\end{equation}
To compute the variation of $F_n$ we need an extra condition
on solutions.

\begin{defn}\label{d:anomalous}
A weak solution $X$ satisfies the \emph{energy inequality}
on $[0,T]$ if
\begin{equation}\label{e:no_anom_diss_condition}
E(t) + \int_0^t\sum_{n\geq1}\frac2{g_n}k_nX_n^2(s)ds
\leq E(0)
,\qquad t\in[0,T].
\end{equation}
A weak solution satisfies the \emph{energy equality}
if there is equality in the above formula.
\end{defn}

We remark that, as is expected in this class of problems,
regularity readily implies uniqueness and that the energy
inequality holds (there is no anomalous dissipation).
The vice versa is not true in general (see for instance~\cite{BarFlaMor2010CRAS,BarMor2013nodea}).

By~\eqref{e:E_n'} and~\eqref{e:no_anom_diss_condition} it follows that,
if $X$ satisfies the energy inequality, then
\begin{equation}\label{e:Fn_deriv}
F_n(t)
  \leq F_n(0)
    + \int_0^t 2\phi_{n-1}k_{n-1}X_{n-1}^2X_n\,ds
    - \int_0^t \sum_{i\geq n}\frac2{g_i}k_iX_i^2\,ds.
\end{equation}

The following proposition gives a sufficient condition for the energy equality.
\begin{prp}\label{p:L^3W^3_no_anom_diss}
Let $T>0$ and $X$ be a weak solution with initial condition $x\in
H$. If $X\in L^3([0,T];W^{\beta/3,3})$, then $X$ satisfies the
energy equality on $[0,T]$.
\end{prp}

\begin{proof}
Let $t\in[0,T]$ and $n\geq1$. By equation~\eqref{e:E_n'},
\begin{equation}\label{star}
0
\leq E_n(t)+\int_0^t\sum_{i\leq n}\frac2{g_i}k_iX_i^2(u)du
=E_n(0)-\int_0^t2\phi_nk_nX_n^2(u)X_{n+1}(u)du.
\tag{$\star$}
\end{equation}
To prove the energy equality, it is sufficient to
take the limit for $n\to\infty$ and show that the last term
of~\eqref{star} converges to zero. By Young's inequality,
\[
\int_0^tk_nX_n^2(u)|X_{n+1}(u)|du
\leq\frac23\int_0^tk_n|X_n(u)|^3du
+\frac13\int_0^tk_{n+1}|X_{n+1}(u)|^3du,
\]
and the terms on the right--hand side converge to zero,
since $X\in L^3([0,T];W^{\beta/3,3})$.
\end{proof}

\subsection{Local existence and uniqueness}

For all $s\in\R$ and $p\geq1$, let $W^{s,p}$ denote the Banach space
\[
  W^{s,p}
    = \bigl\{x=(x_n)_{n\geq1}\in\R^\N:
      \|x\|_{W^{s,p}}^p:=\sum_{n\geq1}2^{psn}|x_n|^p<\infty\bigr\}.
\]
In particular, we set  $H^s = W^{s,2}$ and $H:=H^0=\ell^2(\R)$.

\begin{prp}\label{p:local_exist_uniq}
Let $s>0$ and suppose $x\in H^s$, $g\in H^{-s}$. Then there exist
$\eta>0$, depending only on $\|x\|_{H^s}$, and a unique solution
in the class $\mathcal H:=L^\infty([0,\eta];H^s)$.
\end{prp}
\begin{proof}
In view of applying Banach's fixed point theorem, we introduce the
operator $\mathcal F$ on $\mathcal H$ defined as follows. For all
$n\geq1$ and $t\in[0,\eta]$, let
\[
(\mathcal FV)_n(t)
:=x_ne^{-\frac{k_n}{g_n}t}+\int_0^te^{-\frac{k_n}{g_n}(t-u)}\Bigl[\phi_{n-1}k_{n-1}V_{n-1}^2(u)-\phi_nk_nV_n(u)V_{n+1}(u)\Bigr]du,
\]
so that $X$ is a solution if and only if it is a fixed point of
$\mathcal F$.  To apply Banach's fixed point theorem we must show that
$\mathcal F$ maps some ball $B_{\mathcal H}(M):=\{v\in\mathcal
H:\|v\|_{\mathcal H}\leq M\}$ into itself and that $\mathcal F$ is a
contraction on the ball.
To this end, we will often use that if $v\in\mathcal H$, then $|v_n(t)|\leq
k_n^{-s}\|v\|_{\mathcal H}$ for all $t\geq0$, $n\geq1$.

We deal with the first requirement, so suppose $V\in B_{\mathcal
  H}(M)$. For all $n\geq1$ and $t\in[0,\eta]$,
\begin{align*}
|(\mathcal FV)_n(t)|
&\leq |x_n|e^{-\frac{k_n}{g_n}t}+\|\phi\|_{\ell^\infty}\int_0^te^{-\frac{k_n}{g_n}(t-u)}\Bigl[k_{n-1}V_{n-1}^2+k_n|V_nV_{n+1}|\Bigr]du\\
&\leq |x_n|e^{-\frac{k_n}{g_n}t}+\|\phi\|_{\ell^\infty}(k_{n-1}^{1-2s}+k_n^{1-s}k_{n+1}^{-s})\|V\|^2\int_0^te^{-\frac{k_n}{g_n}(t-u)}du\\
&\leq |x_n|+2\|\phi\|_{\ell^\infty}k_{n-1}^{-2s}g_n(1-e^{-\frac{k_n}{g_n}t})\|V\|^2,
\end{align*}
so $\|\mathcal FV\|_{\mathcal H}\leq \|x\|_{H^s}+2\|\phi\|_{\ell^\infty}\|V\|^2L(\eta)$,
where we defined
\[
L(t)
:=\biggl[\sum_{n\geq1}k_n^{2s}k_{n-1}^{-4s}g_n^2(1-e^{-\frac{k_n}{g_n}t})^2\biggr]^{1/2},
\]
and $\sup_{0\leq t\leq\eta}L(t)=L(\eta)$ by monotonicity. We claim
that $\lim_{\eta\to0}L(\eta)=0$. Consider
\[
L^2(\eta)
=2^{4\beta s}\sum_{n\geq1}k_n^{-2s}g_n^2(1-e^{-\frac{k_n}{g_n}\eta})^2
\leq2^{4\beta s}\sum_{n=1}^Nk_n^{-2s}g_n^2\biggl(\frac{k_n}{g_n}\eta\biggr)^2+2^{4\beta s}\sum_{n>N}k_n^{-2s}g_n^2.
\]
Since $g\in H^{-s}$, we can choose $N$ such that the second term is
arbitrarily small, and then choose $\eta$ in such a way that the first
term is small too, hence $L(\eta)\to0$ as $\eta\to0$.

Let $M:=2\|x\|_{H^s}$. If $\eta$ is small enough so that
$L(\eta)\leq(4\|\phi\|_{\ell^\infty}M)^{-1}$, then
$\|\mathcal FV\|_{\mathcal H}\leq M/2+2\|\phi\|_{\ell^\infty}M^2L(\eta)\leq M$,
so the first requirement is satisfied for all $\eta$ such that
$L(\eta)\leq(8\|\phi\|_{\ell^\infty}\|x\|_{H^s})^{-1}$.

To prove that $\mathcal F$ is a contraction, suppose $X,Y\in B_{\mathcal
  H}(M)$.  For all $n\geq1$ and $t\in[0,\eta]$,
\[
|(\mathcal FX-\mathcal FY)_n(t)|
\leq\|\phi\|_{\ell^\infty}\int_0^te^{-\frac{k_n}{g_n}(t-u)}\Bigl[k_{n-1}\bigl|X_{n-1}^2-Y_{n-1}^2\bigr|+k_n\bigl|X_nX_{n+1}-Y_nY_{n+1}\bigr|\Bigr]du.
\]
With the obvious decomposition
$ab-cd=\frac12(a-c)(b+d)+\frac12(b-d)(a+c)$ and recalling that for all
$j$, $|v_j(t)|\leq k_j^{-s}\|v\|_{\mathcal H}$, we get
\begin{align*}
|(\mathcal FX-\mathcal FY)_n(t)|
&\leq2M\|\phi\|_{\ell^\infty}\|X-Y\|_{\mathcal H}\Bigl[k_{n-1}k_{n-1}^{-2s}+k_nk_n^{-s}k_{n+1}^{-s}\Bigr]\int_0^te^{-\frac{k_n}{g_n}(t-u)}du\\
&\leq4M\|\phi\|_{\ell^\infty}\|X-Y\|_{\mathcal H}k_{n-1}^{-2s}g_n(1-e^{-\frac{k_n}{g_n}t}),
\end{align*}
hence
\[
\|\mathcal FX-\mathcal FY\|_{\mathcal H}
\leq4M\|\phi\|_{\ell^\infty}\|X-Y\|_{\mathcal H}L(\eta).
\]
Let $\theta\in(0,1)$. Choose $\eta$ small enough that
$L(\eta)\leq\theta(8\|\phi\|_{\ell^\infty}\|x\|_{H^s})^{-1}$. Then the
first requirement is satisfied and
$\|\mathcal FX-\mathcal FY\|_{\mathcal H}\leq\theta\|X-Y\|_{\mathcal H}$,
and we conclude by Banach's fixed point theorem.
\end{proof}

\section{The main result}

In this section we prove our main result. The theorem follows immediately
from our Theorem~\ref{thm:main_gen}, which works in a slightly more
general setting.

\begin{thm}\label{thm:main}
Suppose that $g_n$ is non-decreasing, $\frac{k_n}{g_n}$ is eventually
non-decreasing and that $\sum_{n\geq1}g_n^{-1}=\infty$. If $x\in H^s$
for all $s>0$, then there exists a solution $X$ with initial condition
$x$ such that $X\in L^\infty([0,\infty);H^s)$ for all $s>0$. This
  solution is unique in the class
  $L^3_{\textup{loc}}([0,\infty);W^{\beta/3,3})$.
\end{thm}

\subsection{The bounding sequence}

For all initial condition in $H$, we introduce a sequence of positive
numbers which will be fundamental to bound all weak solutions.

\begin{defn}\label{def:bound_seq}
A sequence $y=(y_n)_{n\geq1}$ is the bounding sequence for $x\in H$ if
it is defined by
\begin{align}
y_1&:=y_2:=2\|x\|_H^2,
\\
y_{n+2}
&:=C_{n+2}(y_{n+1}^{1/2})y_n+\sum_{i\geq n+2}x_i^2,
\qquad n\geq1,\label{e:recursion_y}
\end{align}
where for $n\geq3$, $C_n:\R_+\mapsto(0,1)$ is the following increasing
function,
\[
C_n(v)
:=\biggl(1+\frac1{\frac12g_n\|\phi\|_{\ell^\infty}v}\biggr)^{-1}
,\qquad v>0.
\]
\end{defn}

\begin{lem}\label{l:X^2<=y_n}
Suppose $g$ is non-decreasing. Let $T>0$, $x\in H$ and $y$ be the
bounding sequence for $x$. Suppose $X$ is a weak solution with initial
condition $x$ that satisfies the energy inequality on $[0,t]$ for all
$t<T$.
Then $X_n^2(t)\leq y_n$ for all $t\in[0,T)$ and all $n\geq1$.
\end{lem}
\begin{proof}
Define
\begin{equation}\label{e:formula_dn}
d_n^2(t)=F_n(t)+\sum_{i\geq n+1}\int_0^t\frac2{g_i}k_iX_i^2(s)ds.
\end{equation}
Notice that $X_n^2(t)\leq F_n(t)\leq d_n^2(t)\leq\|x\|_H^2$ by the
definition of $d_n$ and~\eqref{e:no_anom_diss_condition}.
By~\eqref{e:Fn_deriv} we deduce that
\begin{equation}\label{e:def_dn}
d_n^2(t)
\leq\int_0^t2\phi_{n-1}k_{n-1}X_{n-1}^2(s)X_n(s)ds-\int_0^t\frac2{g_n}k_nX_n^2(s)ds+F_n(0).
\end{equation}
Define
\[
\bar d_n:=\sup_{0\leq t<T}d_n(t)<\infty.
\]
We claim that for all $n\geq1$
\begin{equation}\label{e:dn_bound_claim}
\bar d_{n+2}^2
\leq C_{n+2}(\bar d_{n+1})\bar d^2_n+F_{n+2}(0).
\end{equation}
Then, since $y_1:=y_2:=2\|x\|_H^2$ and since $C_n$ is monotone
increasing, an easy induction argument yields $\bar d_n^2\leq y_n$
for all $n\geq1$ and hence that
\[
X_n^2(t)
\leq d_n^2(t)
\leq \bar d_n^2
\leq y_n,
\]
for all $n$ and all $t$.

We turn to the proof of the claim~\eqref{e:dn_bound_claim}.
By the Cauchy-Schwarz inequality applied to~\eqref{e:def_dn}
\begin{multline*}
d_n^2(t)
\leq\int_0^t\|\phi\|_{\ell^\infty}k_{n-1}|X_{n-1}(s)|(X_{n-1}^2(s)+X_n^2(s))ds+F_n(0)\\
\leq\|\phi\|_{\ell^\infty}\bar d_{n-1}\int_0^tk_{n-1}(X_{n-1}^2(s)+X_n^2(s))ds+F_n(0)\\
\leq g_n\|\phi\|_{\ell^\infty}\bar d_{n-1}\int_0^t\left(\frac{k_{n-1}}{g_{n-1}}X_{n-1}^2(s)+\frac{k_n}{g_n}X_n^2(s)\right)ds+F_n(0),
\end{multline*}
where we used the fact that $g_n$ and $k_n$ are non-decreasing with
$n$. We get another bound from~\eqref{e:formula_dn},
\begin{multline*}
d_n^2(t)-d^2_{n-2}(t)
=F_n(t)-F_{n-2}(t)-\int_0^t\frac{2k_{n-1}}{g_{n-1}}X_{n-1}^2(s)ds-\int_0^t\frac{2k_n}{g_n}X_n^2(s)ds\\
\leq-2\int_0^t\left(\frac{k_{n-1}}{g_{n-1}}X_{n-1}^2(s)+\frac{k_n}{g_n}X_n^2(s)\right)ds,
\end{multline*}
hence putting the former into the latter,
\[
d_n^2(t)
\leq d^2_{n-2}(t)-2\int_0^t\left(\frac{k_{n-1}}{g_{n-1}}X_{n-1}^2(s)+\frac{k_n}{g_n}X_n^2(s)\right)ds
\leq d^2_{n-2}(t)-\frac{d_n^2(t)-F_n(0)}{\tfrac12g_n\|\phi\|_{\ell^\infty}\bar d_{n-1}},
\]
yielding
\[
d_n^2(t)
\leq \biggl(1+\frac1{\frac12g_n\|\phi\|_{\ell^\infty}\bar d_{n-1}}\biggr)^{-1}d^2_{n-2}(t)+F_n(0)
=C_n(\bar d_{n-1})d^2_{n-2}(t)+F_n(0).
\]
Taking the sup for $s\in[0,T)$ yields the claimed inequality~\eqref{e:dn_bound_claim}.
\end{proof}

Lemma~\ref{l:X^2<=y_n} states that the variables $X_n(t)$ can be bounded
by the the bounding sequence $y$, so we will spend the rest of the section 
to show exponential decay for the bounding sequence $y_n$. 
As a first step we see that bounding sequences converge to $0$.

\begin{lem}\label{l:yn_goes_zero}
Suppose $g$ is non-decreasing and $\sum_{n\geq1}g_n^{-1}=\infty$. Let
$x\in H^s$ for some $s>0$ and let $y$ be the bounding sequence for
$x$. For all $n\geq1$, let $h_n:=\sum_{j\geq n}\sum_{i\geq
  j}x_i^2$. Then
\begin{equation}\label{e:bound_y_product}
y_{n+2m}
\leq y_n\prod_{i=1}^mC_{n+2i}(y_{n+2i-1}^{1/2})+h_n
,\qquad\text{for all } n\geq1, m\geq0.
\end{equation}
Moreover $y_n\to0$ as $n\to\infty$.
\end{lem}

\begin{proof}
Since $C_j\leq1$ for all $j$, inequality~\eqref{e:bound_y_product} is
easily proved by induction on $m$ using~\eqref{e:recursion_y}.

From this we deduce that $y$ is bounded. Since $v\mapsto C_j(v)$
is monotone increasing, we may replace the bound for $y$ inside
$C_j$ yielding that
\[
C_j(y_{j-1}^{1/2})\leq(1+cg_j^{-1})^{-1}
,\qquad j\geq1,
\]
for some constant $c>0$. Since $\sum_{j\geq1}g_j^{-1}=\infty$, then
$\prod_{j\geq1}(1+cg_j^{-1})^{-1}=0$. Since $g$ is monotone, then
$\prod_{i\geq1}(1+cg_{n+2i}^{-1})^{-1}=0$ too, hence
by considering~\eqref{e:bound_y_product} for $n$ and $n+1$, we get,
\[
\limsup_{j\geq n}y_j\leq h_n+h_{n+1}.
\]
Since $x\in H^s$, $\lim_{n\to\infty}h_n=0$, therefore $y_n\to0$.
\end{proof}

The next step is to introduce in Definition~\ref{def:special_subs} below
a special sub--sequence of the indices of $g_n$, this step is necessary
because the hypothesis $\sum_n g_n=\infty$ does not provide enough
information on the rate of divergence of the series.

\begin{defn}\label{def:special_subs}
Given a sequence $g$ with $\sum_{n\geq1}g_n^{-1}=\infty$, a positive 
integer $n_0$ and real numbers $\theta>0,s>0$, define by induction
on $k\geq0$,
\begin{equation}\label{e:def_nk}
n_{k+1}:=\inf\Bigl\{n\geq n_k+2:\sum_{j=n_k+2}^ng_j^{-1}\geq2^{-sk}\theta\Bigr\}<\infty.
\end{equation}
\end{defn}

Notice that the definition above gives a finite number, because
$\sum_{j\geq1}g_j^{-1}=\infty$. The importance of this definition
will be clear with the next two lemmas.

\begin{lem}\label{l:bound_yn_nk}
Suppose $g$ is non-decreasing and $\sum_{n\geq1}g_n^{-1}=\infty$. Let
$x\in H^s$ for some $s>0$ and let $y$ be the bounding sequence for
$x$. Then there exist $n_0\geq1$ and $\theta>0$ such that 
the sequence $(n_k)_{k\geq0}$ given in Definition~\ref{def:special_subs} 
satisfies the following inequality:
\begin{equation}\label{e:sup_yjn}
\sup_{j\geq n_k}y_j
\leq 2^{-2sk}
,\qquad k\geq0.
\end{equation}
\end{lem}
\begin{proof}
In view of applying Lemma~\ref{l:yn_goes_zero}, we need to bound $y_n$
and $h_n$ for $n$ large.  Since $x\in H^s$, then for any $\epsilon>0$,
$x_n\leq\epsilon2^{-sn}$ eventually and in particular for $n$ large,
\[
h_n
=\sum_{j\geq n}\sum_{i\geq j}x_i^2
=\sum_{j\geq 1}jx_{j+n-1}^2
\leq2^{-2sn}\epsilon\sum_{j\geq 1}j2^{-2s(j-1)}.
\]
so for any $\eta>0$, $h_n\leq\eta2^{-2sn}$ eventually. We also know
from Lemma~\ref{l:yn_goes_zero} that $y_n\to0$ as $n\to\infty$. Thus
fix some $\eta>0$ and let $n_0$ be large enough that for all $n\geq
n_0$,
\begin{equation}\label{e:yn_and_hn_for_large_n}
y_n\leq1
\qquad\text{and}\qquad
h_n\leq\eta2^{-2sn}.
\end{equation}

We now proceed to prove~\eqref{e:sup_yjn} by induction on $k\geq0$.
The initial step is simply given by the definition of $n_0$.

We turn to the induction step. Suppose $\sup_{j\geq n_k}y_j\leq
2^{-2sk}$. Then for $j\geq n_k+1$,
\[
C_j(y_{j-1}^{1/2})^{-1}
:=1+\frac1{\frac12g_j\|\phi\|_{\ell^\infty}y_{j-1}^{1/2}}
\geq1+\frac1{\frac12g_j\|\phi\|_{\ell^\infty}2^{-sk}}
=1+c2^{sk}g_j^{-1},
\]
where $c=2/\|\phi\|_{\ell^\infty}$. By~\eqref{e:bound_y_product} we have
then, for $n\geq n_k-1$,
\[
y_{n+2m}
\leq y_n\prod_{i=1}^m(1+c2^{sk}g_{n+2i}^{-1})^{-1}+h_n,
\]
By the monotonicity of $g$,
\[
\prod_{i=1}^m(1+c2^{sk}g_{n+2i}^{-1})
\geq1+c2^{sk}\sum_{i=1}^mg_{n+2i}^{-1}
\geq1+c2^{sk}\frac12\sum_{j=n+2}^{n+2m}g_j^{-1}.
\]
By the definition of $n_{k+1}$ in~\eqref{e:def_nk}, if $n\leq n_k$
and $n+2m\geq n_{k+1}$ we have
\[
\sum_{j=n+2}^{n+2m}g_j^{-1}
\geq2^{-sk}\theta.
\]
Collecting all conditions, we have proved that if $n\in\{n_k-1,n_k\}$
and $n+2m\geq n_{k+1}$, then
\[
y_{n+2m}
\leq y_n(1+\frac12c\theta)^{-1}+h_n.
\]
Since $n\geq n_k-1\geq n_{k-1}$, then by inductive hypothesis $y_n\leq
2^{-2s(k-1)}$; moreover since $n\geq n_k-1\geq k$, then by the second one
of~\eqref{e:yn_and_hn_for_large_n}, $h_n\leq\eta2^{-2sk}$, so the
bound above becomes
\[
y_{n+2m}
\leq2^{-2s(k-1)}(1+\frac12c\theta)^{-1}+\eta2^{-2sk}.
\]
Now we choose $\theta$ large enough and $\eta$ small enough that
\[
2^{2s}(1+\frac12c\theta)^{-1}+\eta
\leq2^{-2s},
\]
to get
\[
y_{n+2m}
\leq 2^{-2s(k+1)}
,\qquad n\in\{n_k-1,n_k\},\quad n+2m\geq n_{k+1}.
\]
Since for all $j\geq n_{k+1}$ there
exist $n$ and $m$ such that $n_k-1\leq n\leq n_k$ and $j=n+2m\geq n_{k+1}$,
we have proved
\[
\sup_{j\geq n_{k+1}}y_j
\leq2^{-2s(k+1)},
\]
closing the induction.
\end{proof}

\begin{lem}\label{l:nk+1=2+nk}
Suppose $g$ is non-decreasing and $\sum_{n\geq1}g_n^{-1}=\infty$.  Let
$n_0\geq1$ and $\theta>0$ be constant. If $(n_k)_{k\geq0}$ is as in
Definition~\ref{def:special_subs} then there exist infinitely many
$k\geq1$ such that $n_{k+1}=n_{k}+2$.
\end{lem}
\begin{proof}
Suppose that there exists a non--negative integer $r$ such that
$n_{k+1}\geq n_k+3$ for all
 $k\geq r$.  By the definition of the
sequence $(n_k)_{k\geq0}$, we know that for $k\geq r$,
\[
\sum_{j=n_k+2}^{n_{k+1}-1}g_j^{-1}<2^{-sk}\theta.
\]
Summing on $k$ we obtain
\[
\sum_{k\geq r}\sum_{j=n_k+2}^{n_{k+1}-1}g_j^{-1}<\infty,
\]
hence since $\sum_{j\geq n_r}g_j^{-1}=\infty$,
\begin{equation}\label{e:divergent}
\sum_{k\geq r}(g_{n_k}^{-1}+g_{n_k+1}^{-1})=\infty.
\end{equation}
But $g_{n_k+1}^{-1}
\leq g_{n_k}^{-1}
\leq g_{n_{k-1}+2}^{-1}
\leq 2^{-s(k-1)}\theta$,
which is in contradiction with~\eqref{e:divergent}. Hence there exist
infinitely many $k$ such that $n_{k+1}=n_k+2$.
\end{proof}

\begin{lem}\label{l:bound_y_n}
Let $x\in H^s$ for some $s>0$ and let $y$ be the bounding sequence for
$x$. Suppose that $g_n$ is non-decreasing, $g_n2^{-sn}$ is eventually
non-increasing and that $\sum_{n\geq1}g_n^{-1}=\infty$.
Then 
\[
\sum_{n\geq1} 2^{2sn}y_n< \infty.
\]
\end{lem}
\begin{proof}
Let us recall the recursion~\eqref{e:recursion_y} that defines the
bounding sequence $y$,
\begin{equation}\label{e:y_n_recurs_bound}
y_{n+2}:=c_{n+2}y_n+f_{n+2}
,\qquad n\geq1,
\end{equation}
where
\[
c_n
:=C_n(y_{n-1}^{1/2})
:=\left(1+\frac1{\frac12g_n\|\phi\|_{\ell^\infty}y_{n-1}^{1/2}}\right)^{-1}
,\qquad n\geq3,
\]
and where $f_n:=F_n(0)$, $n\geq 3$.
Since $x\in H^s$, it is immediate that
\begin{equation}\label{e:f_n_H_s}
\sum_{n\geq1}2^{2sn}f_n<\infty,
\end{equation}
so our strategy will be to show that
$c_n\to0$ as $n\to\infty$.

By Lemma~\ref{l:bound_yn_nk} there exist $n_0$ and $\theta$ such that
the sequence $(n_i)_{i\geq1}$ of Definition~\ref{def:special_subs}
satisfies
\begin{equation}\label{e:bound_y_j}
\sup_{j\geq n_i} y_j\leq 2^{-2si}
,\qquad i\geq0.
\end{equation}
\emph{A fortiori} these inequalities hold also if we take larger values for
$n_0$ and $\theta$, so let $\theta$ be large enough to verify
inequality~\eqref{e:bound theta} below and let $n_0$ be large enough
that:
\begin{enumerate}
\item $f_n\leq\frac12 2^{-2sn}$ for $n\geq n_0$ (a consequence of~\eqref{e:f_n_H_s});
\item $n\mapsto g_n2^{-sn}$ is non-increasing for $n\geq n_0$.
\end{enumerate}

By Lemma~\ref{l:nk+1=2+nk} there exists $k$ such that
$n_{k+1}=2+n_k$, that is, $g^{-1}_{n_k+2}\geq2^{-sk}\theta$ hence
\begin{equation}\label{e:g_n_k+m}
g_{n_k+m}\leq\frac1{2^{2s}\theta}2^{s(k+m)}
,\qquad m\geq2.
\end{equation}

We have all the ingredients to prove the following inequality:
\begin{equation}\label{e:claim_ynkm}
 y_{n_k+m}\leq2^{2s}2^{-2s(k+m)}
,\qquad m\geq0.
\end{equation}
Let us proceed by induction on $m$. The initial steps for $m=0$ and $m=1$
follow immediately from~\eqref{e:bound_y_j} with $i=k$.

For the inductive step, suppose the inequality~\eqref{e:claim_ynkm}
is true up to $m+1$. By~\eqref{e:g_n_k+m} and the inductive hypothesis,
\[
c_{n_k+m+2}
=\left(1+\frac1{\frac12g_{n_k+m+2}\|\phi\|_{\ell^\infty}y_{n+k+m+1}^{1/2}}\right)^{-1}
\leq\left(1+\frac{2\theta}{\|\phi\|_{\ell^\infty}}\right)^{-1}
\leq \frac12 2^{-4s},
\]
if we choose $\theta$ large enough that
\begin{equation}\label{e:bound theta}
\left(1+\frac{2\theta}{\|\phi\|_{\ell^\infty}}\right)^{-1}
\leq \frac12 2^{-4s}.
\end{equation}
Moreover, since $n_k\geq k-1$, we have
\[
f_{n_k+m+2}
\leq\frac122^{-2s(n_k+m+2)}
\leq\frac122^{2s}2^{-2s(k+m+2)},
\]
hence
\[
y_{n_k+m+2}
=c_{n_k+m+2}y_{n_k+m}+f_{n_k+m+2}
\leq2^{2s}2^{-2s(k+m+2)},
\]
thus closing the induction.

Inequality~\eqref{e:claim_ynkm} says us that $y_n\to0$ at least as
fast as $2^{-2sn}$. To get $\sum_n 2^{2sn}y_n<\infty$ we need a little
bit more. We proved above that for any $\theta$ large enough there
exists $n_k$ such that
\[
\sup_{j\geq n_k}c_j
\leq\left(1+\frac{2\theta}{\|\phi\|_{\ell^\infty}}\right)^{-1}.
\]
By the arbitrarity of $\theta$, $\lim_{n\to\infty}c_n=0$. This
together with~\eqref{e:y_n_recurs_bound} and~\eqref{e:f_n_H_s} proves
that
\[
\sum_{n\geq1}2^{2sn}y_n<\infty.\qedhere
\]
\end{proof}

\subsection{Global existence, uniqueness and regularity}

\begin{thm}\label{thm:main_gen}
Let $x\in H^s$ for some $s>\frac\beta3$. Suppose that $g_n\in H^{-s}$
is non-decreasing, $g_n2^{-sn}$ is eventually non-increasing and that
$\sum_{n\geq1}g_n^{-1}=\infty$.  Then there exists a solution in the
class $L^\infty([0,\infty);H^s)$. This solution is unique in the class
$L^3_{\textup{loc}}([0,\infty);W^{\beta/3,3})$.
\end{thm}
\begin{proof} 
Let $T>0$ be the maximal time of existence in $H^s$ of the
solution provided by Proposition~\ref{p:local_exist_uniq}.
In particular, $X\in L^\infty([0,t];H^s)$ for all
$t<T$ and, since $s>\beta/3$, 
$X\in L^3([0,t];W^{\beta/3,3})$ for all $t<T$.
Hence, by Proposition~\ref{p:L^3W^3_no_anom_diss}, $X$ satisfies
the energy equality on $[0,t]$.
Lemma~\ref{l:X^2<=y_n} applies, so if $y$ denotes the bounding
sequence for $x$, we have
\begin{equation}\label{e:Xn2t<=yn}
X_n^2(t)\leq y_n
,\qquad n\geq1,\quad t\in[0,T).
\end{equation}
By Lemma~\ref{l:bound_y_n}
\[
\sup_{t\in[0,T)}\|X(t)\|_{H^s}^2
\leq\sum_{n\geq1}2^{2sn}y_n<\infty.
\]
If $T=\infty$ we just proved $X\in L^\infty([0,\infty);H^s)$. Suppose
  by contradiction that $T<\infty$. Then the bound
  in~\eqref{e:Xn2t<=yn} can be extended to $t\in[0,T]$ by the
  continuity of $X_n$ hence again by Lemma~\ref{l:bound_y_n}, $X(T)\in
  H^s$ and it would be possible to apply
  Proposition~\ref{p:local_exist_uniq}, in contradiction with
  the maximality of $T$.

Finally we turn to uniqueness in
$L^3_{\textup{loc}}([0,\infty);W^{\beta/3,3})$. By
  Proposition~\ref{p:L^3W^3_no_anom_diss}, Lemma~\ref{l:X^2<=y_n} and
  Lemma~\ref{l:bound_y_n}, if $X$ is a solution of class
  $L^3_{\textup{loc}}([0,\infty);W^{\beta/3,3})$, then $X$ is also of
    class $L^\infty_{\textup{loc}}([0,\infty);H^s)$, hence by
      Proposition~\ref{p:local_exist_uniq} it is unique.
\end{proof}
\subsection{Additional remarks}\label{s:remarks}

The last part of the paper is devoted to some final remarks about
our results. They have been collected here in order to give a 
more complete understanding of the problem,
while focusing, in the main body of the paper, on the assumptions
corresponding to those of~\cite{Tao2009}.
%
\subsubsection{A useful generalization}
\label{s:generalization_to_nonconstant_phi}

The results presented in the previous sections allow for more general
coefficients $\phi$. Namely, assume that
\begin{equation}\label{e:gen}
  \phi_n
    = \phi_n(t, X_{n-m},X_{n-m+1},\dots,X_{n+m}),
\end{equation}
for all $n\geq1$, where $m\geq1$ is a fixed integer. For convenience
we set $X_{-m}=X_{-m+1}=\dots=X_0=0$.
Assume moreover that the functions $(\phi_n)_{n\geq1}$ are uniformly bounded
and uniformly Lipschitz. This ensures that the local existence and
uniqueness theorem (Proposition~\ref{p:local_exist_uniq}) still holds.
In Proposition~\ref{p:L^3W^3_no_anom_diss} and Lemma~\ref{l:X^2<=y_n}
we only use the uniform boundedness, while
lemmas~\ref{l:yn_goes_zero},~\ref{l:bound_yn_nk}-\ref{l:bound_y_n}
deal only with bounding sequences.

The above model has a nice application to the averaged Navier-Stokes
system studied by Tao in~\cite{Tao2014}. By making a special average
on the trasport of the NS equations, the author derives a
vector-valued dyadic system, very similar to~\eqref{e:main} but with
four component for each $n$. A general version of this averaged system
is
\begin{equation}\label{e:averaged_NS}
  \begin{cases}
    X_{1,n}'
      = -\frac{k^\alpha_n}{g_n}X_{1,n}
      +k^\gamma_{n}\left(-C_1X_{3,n}X_{4,n}-C_2X_{1,n}X_{2,n}
      -C_3X_{1,n}X_{3,n}+C_4X^2_{4,n-1}\right),\\
    X_{2,n}'
      = -\frac{k^\alpha_n}{g_n}X_{2,n}
      +k^\gamma_{n}\left(C_2X^2_{1,n}-C_5X^2_{3,n} \right),\\
    X_{3,n}'
      = -\frac{k^\alpha_n}{g_n}X_{3,n}
      +k^\gamma_{n}\left(C_3X^2_{1,n}+C_5X_{2,n}X_{3,n}\right),\\
    X_{4,n}'
      = -\frac{k^\alpha_n}{g_n}X_{4,n}
      +k^\gamma_{n} C_1X_{1,n}X_{3,n}
      -k^\gamma_{n+1} C_4X_{4,n}X_{1,n+1},\\
    X_{\cdot,n}(0)
      = x_{\cdot,n},
  \end{cases}
\end{equation}
for all $t>0$ and $n\geq1$.

Here $X=(X_{i,n})_{i\in\{1,2,3,4\},n\geq1}$ is a family of real
functions, $X_{i,n}:[0,\infty)\to\mathbb R$; $X_{\cdot,0}\equiv0$;
  $x=(x_{i,n})_{i\in\{1,2,3,4\},n\geq1}$ is the given initial
  condition, $k_n=2^{\beta n}$ with $\beta>0$, and $C_1,\ldots,C_5$
  are five real constant.

In the framework of Navier-Stokes equations the constants $\alpha=2$
and $\gamma=\frac{5}{2}$ give a strictly supercritical regime.
In~\cite{Tao2014} the author shows that this system with a suitable
initial condition develops a singularity.  For
system~\eqref{e:averaged_NS} the critical regime is for
$\alpha=\gamma$ and $g\equiv 1$ (it is the regime in which the
trasport effects are of the same order of the dissipative effect)
whereas the logarithmically supercritical regime congectured
in~\cite{Tao2009} is given by $\alpha=\gamma$ and $g$ such that
$\sum_n g_n^{-1}=\infty$.

The latter case can be included in our model~\eqref{e:main} with
general coefficients~\eqref{e:gen}. Indeed, by summing up the
components $X^2_n:=\sum_{i=1}^4 X^2_{i,n}$ one gets
\begin{align*}
\frac12\frac d{dt}X_n^2
&=-\frac{k^\alpha_n}{g_n}X_n^2+C_4(k_n^\gamma X_{4,n-1}^2X_{1,n}-k_{n+1}^\gamma X_{4,n}^2X_{1,n+1})\\
&=-\frac{k^\alpha_n}{g_n}X_n^2+\phi_nk_n^\gamma X_{n-1}^2X_n-\phi_{n+1}k_{n+1}^\gamma X_n^2X_{n+1}
\end{align*}
and this can be reduced to the system~\eqref{e:main}, when
$\alpha=\gamma$ by a suitable choice of $\beta$ and appropriate
functions $\phi_n(t)$ depending on $n$ and $t$ and uniformly bounded.

\subsubsection{Conditions for smoothing}

Here we study the smoothing effect of the dissipative part. We work
under the assumptions of Theorem~\ref{thm:main}.
\paragraph{The linear operator.}

Consider the system $Z_n' = -\frac{k_n}{g_n}Z_n$, $n\geq1$,
the linear part of~\eqref{e:main}.
\begin{lem}
  Assume additionally that $\frac{ng_n}{k_n}\to0$. If $x\in\ell^2$
  and $Z$ is the solution starting at $x$, then
  $Z(t)\in L^\infty([\epsilon,\infty);H^s)$ for every
  $\epsilon>0$ and every $s>0$.
\end{lem}
\begin{proof}
  Clearly $Z_n(t) = x_n\exp\bigl(-\frac{k_n}{g_n}t\bigr)$
  and $\sup_n (2^{sn}Z_n(t))<\infty$ for all $s>0$, $t>0$
  if and only if $\frac{ng_n}{k_n}\to0$.
\end{proof}
\begin{rem}
  If $\frac{ng_n}{k_n}\not\to0$, the linear dissipation may not
  have a smoothing effect. Indeed, it is easy to construct
  a counterexample. Choose $n_1\geq1$ and set
  $n_{p+1} = n_p2^{\beta k_{n_p}/n_p}$, $p\geq1$,
  $g_{n_p} = k_{n_p}/n_p$, and define
  $g_n = g_{n_p}$ for $n_p\leq n\leq m_p$,
  and $g_n = g_{n_p}k_n/k_{m_p}$ for $m_p+1\leq n< n_{p+1}$,
  where $m_p = n_p + k_{n_p}/n_p$ It is easy to verify
  that $(g_n)_{n\geq1}$ satisfies our standing assumptions and
  that there are sequences $(x_n)_{n\geq1}\in\ell^2$ such that the
  corresponding solution $Z$ is not smooth.
\end{rem}
\paragraph{Smoothing by dissipation.}

We now analyse the smoothing effect for the non--linear equation.
Our final result is the following.
\begin{thm}
  Assume additionally that $\frac{ng_n}{k_n}\to0$. If $s>\beta$
  and $X$ is a solution such that $X(0)\in H^s$ and
  $X\in L^\infty([0,T];H^s)$, then
  $X\in L_\textup{loc}^\infty((0,T];H^s)$ for every $s>0$.
\end{thm}
The theorem follows immediately from the following lemma.
\begin{lem}
  Under the same assumptions of the previous theorem,
  let $s_1>\beta$ and $s_2\in(s_1,2s_1-\beta)$.
  If $X$ is a solution such that $X(0)\in H^{s_1}$
  and $X\in L^\infty([0,T];H^{s_1})$, then
  $X\in L_\textup{loc}^\infty((0,T];H^{s_2})$.  
  More precisely, there is a non--decreasing
  upper semi--continuous function $\varphi:(0,\infty)\to\mathbf{R}$
  such that $\varphi$ is continuous in $0$ with $\varphi(0) = 0$,
  and
  \[
    \sup_{t\in[0,T]}\bigl(\varphi(t)\|X(t)\|_{H^{s_2}}\bigr)
      <\infty.
  \]
\end{lem}
\begin{proof}
  We have that
  \[
    2^{s_2n}X_n(t)
      = 2^{s_2n}\operatorname{e}^{-\frac{k_n}{g_n}t}X_n(0)
        + 2^{s_2n}\int_0^t\operatorname{e}^{-\frac{k_n}{g_n}(t-s)}
        (\phi_{n-1}k_{n-1}X_{n-1}^2 - \phi_n k_n X_nX_{n+1})\,ds,
  \]
  and consider the two terms on the right hand side separately.

  For the non--linear term, we use the inequality
  $|X_n(t)|\leq 2^{-s_1n}\|X\|_{L^\infty(H^{s_1})}$
  to get
  \begin{equation}
  \begin{array}{l}
   \Bigl|2^{s_2n}\int_0^t\operatorname{e}^{-\frac{k_n}{g_n}(t-s)}
        (\phi_{n-1}k_{n-1}X_{n-1}^2 - \phi_n k_n X_nX_{n+1})\,ds\Bigr|\leq\\
    \qquad\leq c\|X\|_{L^\infty(H^{s_1})}^2\|\phi\|_{\ell^\infty}2^{n(s_2-2s_1)}g_n
      \leq c\|X\|_{L^\infty(H^{s_1})}^2\|\phi\|_{\ell^\infty}2^{n(s_2-2s_1+\beta)}
      \in\ell^2,     
  \end{array}
  \end{equation}
  since $g_n\leq c k_n$ and, by the choice of $s_2$, $s_2-2s_1+\beta<0$.

  For the term with the initial condition we notice that
  \[
    2^{s_2n}\operatorname{e}^{-\frac{k_n}{g_n}t}|X_n(0)|
      = 2^{(s_2-s_1)n}\operatorname{e}^{-\frac{k_n}{g_n}t}
        \bigl(2^{s_1n}|X_n(0)|\bigr)
      \leq \psi(t)\bigl(2^{s_1n}|X_n(0)|\bigr)
      \in\ell^2,
  \]
  where $\psi(t) = \sup_n \bigl(2^{(s_2-s_1)n}
  \exp\bigl(-\frac{k_n}{g_n}t\bigr)\bigr)$.
  It is easy to check that $\psi$ is non--increasing,
  lower semi--continuous and $\psi(t)\uparrow\infty$
  as $t\downarrow0$. Choose $\varphi = 1/\psi$ to
  conclude the proof.
\end{proof}

\end{document}